\documentclass[11pt]{amsart}
\usepackage[a4paper,margin=28mm]{geometry}
\usepackage[T1]{fontenc}
\usepackage{lmodern,microtype,mathtools,amssymb}
\usepackage[hidelinks]{hyperref}
\hypersetup{pdftitle={A Quadratic Ancestor Inequality for Spanning Trees},
 pdfsubject={Weighted spanning trees and loop-erased random walk},
 pdfkeywords={spanning tree, loop-erased random walk, quadratic ancestor inequality}}
\newtheorem{theorem}{Theorem}[section]
\newtheorem{proposition}[theorem]{Proposition}

\theoremstyle{definition}

\theoremstyle{remark}
\newtheorem{remark}[theorem]{Remark}
\newcommand{\PP}{\mathbb P}
\newcommand{\EE}{\mathbb E}
\newcommand{\RR}{\mathbb R}
\newcommand{\ZZ}{\mathbb Z}
\newcommand{\LE}{\operatorname{LE}}
\newcommand{\LCA}{\operatorname{LCA}}
\newcommand{\Comp}{\operatorname{Comp}}

\newcommand{\Es}{\operatorname{Es}}

\newcommand{\one}{\mathbf 1}
\newcommand{\cB}{\mathcal B}
\newcommand{\cF}{\mathfrak F}
\newcommand{\dd}{\,\mathrm d}
\numberwithin{equation}{section}
\title[A Quadratic Ancestor Inequality for Spanning Trees]{A Quadratic Ancestor Inequality for Spanning Trees}
\date{\today}
\author{Gefei Cai}
\address{Beijing International Center for Mathematical Research,
Peking University}
\email{caigefei1917@pku.edu.cn}
\begin{document}

\begin{abstract}
We introduce a quadratic ancestor inequality for weighted spanning trees on finite undirected graphs and give two applications.
First, for every finite $n$-vertex transitive graph, the expected distance in its uniform spanning tree (UST) between two independent uniform vertices is at least $\sqrt n-1$, resolving a conjecture in [Benjamini--Kozma, \emph{Comm. Math. Phys.} \textbf{259} (2005), 257--286].
Second, we prove that the growth exponent of three-dimensional loop-erased random walk (LERW) is at least $3/2$.
\end{abstract}
\maketitle

\section{Introduction}

Let $G=(V,E)$ be a finite connected undirected graph, with positive edge conductances $(c_e)_{e\in E}$ and a distinguished root $r\in V$.
A weighted spanning tree $T$ has distribution
\begin{equation}\label{eq:tree-law}
 \PP(T=t)=Z_G^{-1}w(t),\qquad
 w(t)=\prod_{e\in t}c_e,\qquad Z_G=\sum_{t\text{ spanning tree}}w(t).
\end{equation}
Write $\gamma_x=[x,r]_T$ for its root path, including both endpoints, and $\LCA_T(x,y)$ for the lowest common ancestor
of $x$ and $y$, namely the first common vertex of $\gamma_x$
and $\gamma_y$ when followed towards $r$.
For $B\subseteq V\setminus\{r\}$, define
\begin{equation}\label{eq:h-a}
 H_B(x)=\PP(\gamma_x\cap B\ne\varnothing),\qquad
 a_B(x,y)=\PP(\LCA_T(x,y)\in B).
\end{equation}
Let $U_B=V\setminus(B\cup\{r\})$.
For distinct $x,y\in U_B$, let $J_B(x,y)$ be the probability that $x,y,r$ belong to three different components of the induced forest $T\setminus B$, and put $J_B(x,x)=0$.
We also write \(H_B=(H_B(x))_{x\in U_B}\) and \(J_B=(J_B(x,y))_{x,y\in U_B}\).

We establish the following inequality for weighted spanning trees.
\begin{theorem}[Quadratic ancestor inequality]\label{thm:separation}
For every $G,r,B$ as above, the matrix
$H_BH_B^{\mathsf T}-J_B$ is positive semidefinite. Consequently,
\begin{equation}\label{eq:ancestor}
 \sum_{x,y\in V}f(x)f(y)a_B(x,y)
 \le \left(\sum_{x\in V}f(x)H_B(x)\right)^2
\end{equation}
for every nonnegative $f:V\to\RR$.
For all $x,y\in V$, one also has
\begin{equation}\label{eq:pointwise}
 a_B(x,y)\le 2H_B(x)H_B(y).
\end{equation}
In particular, if $f$ is a probability distribution on $V$ and
$X,Y$ are independent with common law $f$,
independently of $T$, then \eqref{eq:ancestor} becomes
$\PP(\LCA_T(X,Y)\in B)
\le\PP(\gamma_X\cap B\ne\varnothing)^2$.
\end{theorem}

We prove Theorem~\ref{thm:separation} using completely log-concave matroid polynomials from Anari--Liu--Oveis Gharan--Vinzant~\cite{ALOV} and Br\"and\'en--Huh~\cite{BH}. We also refer to e.g.~\cite{COSW,BBL,Lyons} for related work on generating polynomials of spanning trees
and their negative dependence properties.
The constant $1$ in the inequality \eqref{eq:ancestor} is optimal.
Indeed, take the deterministic tree with edges $\{r,b\}$ and $\{b,x_i\}$, $1\le i\le k$, put $B=\{b\}$, and let $f$ be uniform on the $k$ leaves. The two sides of \eqref{eq:ancestor} are then $1-1/k$ and $1$, respectively.

Theorem~\ref{thm:separation} has the following two main applications. We call a graph \emph{transitive} if its conductance-preserving automorphisms act transitively on its vertices.

\begin{theorem}\label{thm:bk}
Let $G$ be a finite connected transitive weighted graph with $n$ vertices.
Let $T$ have law \eqref{eq:tree-law}, and let $X,Y$ be independent uniform vertices, independent of $T$.
Write $\gamma=[X,Y]_T$.
For every deterministic $B\subseteq V$,
\begin{equation}\label{eq:set-hit}
 \PP(\gamma\cap B\ne\varnothing)\ge\sqrt{\frac{|B|}{n}}.
\end{equation}
For a deterministic partition $\mathcal P=\{B_1,\ldots,B_k\}$ of $V$ into nonempty sets,
let $N_{\mathcal P}$ count the blocks of $\mathcal P$ visited by $\gamma$. Then
\begin{equation}\label{eq:partition-hit}
 \EE N_{\mathcal P}\ge\sum_{i=1}^k\sqrt{\frac{|B_i|}{n}}.
\end{equation}
In particular, let $d_T(X,Y)$ be the number of edges
in the unique path from $X$ to $Y$ in $T$. Then
\begin{equation}\label{eq:bk}
 \EE d_T(X,Y)\ge\sqrt n-1.
\end{equation}
\end{theorem}

Benjamini and Kozma~\cite[Introduction]{BK} asked whether the expected
UST distance between two independent uniform vertices of every finite transitive graph admits a universal lower bound of order $\sqrt n$. Theorem~\ref{thm:bk} then answers this question.
Previously, lower bounds of order $\sqrt n$ were 
obtained under high-dimensional or
model-specific assumptions~\cite{PR,Sch4}. See also~\cite{CHL,MNS,ANSdense,ANSghp,AS} for related results on UST
distances, diameters, and scaling limits.

Our second application is to the growth exponent
of three-dimensional LERW. Let $S$ be a simple random walk on $\ZZ^3$, started at $0$.
For $B_R:=\{x\in\ZZ^3:|x|<R\}$, set
\begin{equation}\label{eq:beta-def}
 \tau_R=\inf\{j\ge0:S_j\notin B_R\},\qquad
 M_R=|\LE(S[0,\tau_R])|,\qquad
 \beta=\lim_{R\to\infty}\frac{\log\EE M_R}{\log R}.
\end{equation}
The existence of the limit was proved by Shiraishi~\cite{ShiG}.
Let $\eta=\LE(S[0,\infty))$.

\begin{theorem}\label{thm:3d}
There is $c>0$ such that, for every $R\ge1$,
\begin{equation}\label{eq:occupation-main}
 \EE|\eta\cap B_R|\ge cR^{3/2}.
\end{equation}
Consequently, we have $\beta\ge3/2$.
\end{theorem}

Note that $\beta$ also equals the Hausdorff dimension of the scaling limit of the three-dimensional LERW~\cite{ShiH}.
The classical bound $\beta>1$ follows by comparison with random-walk
intersection exponents; see Lawler~\cite[Introduction]{Lawler1999} and the
references therein.
Recently, \cite{Liu} made this approach quantitative and improved it to $\beta>1.0001$. Theorem~\ref{thm:3d} also implies that two independent infinite LERWs on $\ZZ^3$ intersect infinitely often
a.s.; see Theorem~\ref{thm:independent-lerw}.

\smallskip
\paragraph{\textbf{Organization of the paper}}
We prove Theorem~\ref{thm:separation} in Section~\ref{sec:deletion}.
Sections~\ref{sec:finite-apps} and~\ref{sec:3d} then prove Theorems~\ref{thm:bk} and~\ref{thm:3d}, respectively.
Section~\ref{sec:further} provides further applications of Theorem~\ref{thm:separation}, including that the three-dimensional long-range LERW in the non-mean-field regime has lower growth exponent at least $3/2$ as well (Proposition~\ref{prop:long-range}).

\smallskip
\paragraph{\textbf{Statement on the use of AI}}
The original motivation for this work was to prove Theorem~\ref{thm:3d}.
The author reduced this problem to the sufficient condition~\eqref{eq:ancestor},
a general inequality for finite graphs. With the assistance of 
AI, the author subsequently developed the proof of Theorem~\ref{thm:separation}
and found that the theorem also yields Theorem~\ref{thm:bk} as a corollary. The author takes full responsibility for the paper.

\section{The quadratic ancestor inequality}\label{sec:deletion}

A polynomial with nonnegative coefficients is \emph{completely
log-concave} if every iterated directional derivative (including order zero) in
nonnegative directions is either zero or has concave logarithm
on the positive orthant.
We use the following input from the complete log-concavity of matroid polynomials~\cite{ALOV,BH}.

\begin{theorem}[Complete log-concavity of matroid polynomials]\label{thm:matroid-input}
Let $\mathsf M$ be a matroid on a finite ground set $F$, with
independent sets $\mathcal I(\mathsf M)$.
Then
\begin{equation}\label{eq:matroid-input}
 \sum_{I\in\mathcal I(\mathsf M)}
 s^{|F|-|I|}\prod_{e\in I}z_e
 \quad\text{and}\quad
 \sum_{\substack{A\subseteq F\\ A\text{ a basis of }\mathsf M}}
 \prod_{e\in A}z_e
\end{equation}
are homogeneous completely log-concave polynomials.
Moreover, the class of homogeneous completely log-concave
polynomials is closed under products, partial derivatives, and
nonnegative linear substitutions.
\end{theorem}
\begin{proof}
The complete log-concavity of the first polynomial is proved in
\cite[Theorem~4.1]{ALOV}; the same property for both polynomials
follows from \cite[Proposition~4.25 and Theorem~2.30]{BH}. The closure properties follow from \cite[Theorem~2.10 and Corollary~2.32]{BH}.
\end{proof}

For clarity, a matroid is a nonempty hereditary family of subsets of a finite ground set satisfying the augmentation axiom; its maximal independent sets are its bases.
We use only graphic matroids, whose independent sets are forests, free matroids, in which every subset is independent, and dual matroids, whose bases are complements of the original ones.

For a realization of the weighted spanning tree $T$, write $\mathcal C_B(T)=\Comp(T\setminus B)$ and let $C_r(T)$ be the component containing $r$.
For variables $z=(z_u)_{u\in U_B}$, define
\begin{equation}\label{eq:Q-def}
 Q_B(z)=\EE \prod_{C\in\mathcal C_B(T)\setminus\{C_r(T)\}}
                  \left(1+\sum_{u\in C}z_u\right).
\end{equation}
An empty product is $1$.

\begin{proposition}\label{prop:Q}
The polynomial $Q_B$ is completely log-concave.
\end{proposition}
\begin{proof}
Put $U=V\setminus B$, $k=|V|-1$, and $E_U=E(G[U])$.
Add a new vertex $*$ and the star edges $M=\{s_u=\{*,u\}:u\in U\}$.
On $E\sqcup M$, take the direct sum $\mathsf N$ of the graphic matroid of $G[U]+M$ and the free matroid on $E\setminus E_U$.
Set $N=|E|+|U|$.

Introduce variables $t_e$ for $e\in E$, variables $z_u$ for $u\in U$, and a homogenizing variable $s$.
The polynomial
\begin{equation}\label{eq:I}
 I(s,t,z)=\sum_{J\in\mathcal I(\mathsf N)}s^{N-|J|}
       \prod_{e\in J\cap E}(c_et_e)\prod_{s_u\in J\cap M}z_u
\end{equation}
is homogeneous completely log-concave by Theorem~\ref{thm:matroid-input} after the change of variables \(t_e\mapsto c_et_e\).
So is the basis generating polynomial of the dual graphic matroid
\begin{equation}\label{eq:B-star}
 B^*(t)=\sum_{T\text{ spanning tree of }G}\prod_{e\in E\setminus T}t_e.
\end{equation}

Extracting the coefficient of \(\prod_{e\in E}t_e\) gives
\begin{equation}\label{eq:coefficient}
 F(s,z)=\left[\prod_{e\in E}t_e\right]\bigl(I(s,t,z)B^*(t)\bigr)
 =\left.\left(\prod_{e\in E}\partial_{t_e}\right)(IB^*)\right|_{t=0}.
\end{equation}
The closure properties above show that $F$ is also homogeneous completely log-concave, of degree $D=N-k$.
For each spanning tree $T$, a subset $S\subseteq M$
can be added to $T\cap E_U$ without creating a cycle
precisely when $S$ contains at most one edge $s_u$
with $u$ in each component of $T\setminus B$.
Consequently,
\begin{equation}\label{eq:F}
 F(s,z)=\sum_T w(T)s^{D-|\mathcal C_B(T)|}
                       \prod_{C\in\mathcal C_B(T)}\left(s+\sum_{u\in C}z_u\right).
\end{equation}
All exponents in~\eqref{eq:F} are nonnegative, since $D\ge |U|\ge|\mathcal C_B(T)|$.

Differentiate $F(s,z)$ with respect to \(z_r\) and then set \(s=1\).
Since only the factor belonging to $C_r(T)$ is differentiated, we have
\begin{equation}\label{eq:Q-extract}
 Q_B(z)=\left.Z_G^{-1}\partial_{z_r}F(s,z)\right|_{s=1}.
\end{equation}
Note that partial differentiation preserves homogeneous complete log-concavity, and specialization at \(s=1\) preserves complete log-concavity in the remaining variables. We conclude.
\end{proof}

\begin{proof}[Proof of Theorem~\ref{thm:separation}]
By \eqref{eq:Q-def}, for every $x,y\in U_B$,
\begin{equation}\label{eq:derivatives}
 Q_B(0)=1,\qquad \partial_xQ_B(0)=H_B(x),\qquad
 \partial_x\partial_yQ_B(0)=J_B(x,y).
\end{equation}
Furthermore, $\partial_x^2 Q_B(z)=0$ as $z_x$ occurs in at most one factor in each product in \eqref{eq:Q-def}.
Hence, by Proposition \ref{prop:Q}, we obtain that 
$-\nabla^2\log Q_B(0)=H_BH_B^{\mathsf T}-J_B$ is positive semidefinite.

For $x,y\in U_B$, the event $\{\LCA_T(x,y)\in B\}$ implies that $x,y,r$ lie in three different components of $T\setminus B$.
Thus $a_B(x,y)\le J_B(x,y)$, and \eqref{eq:ancestor} follows for $f$ supported on $U_B$.
Taking two-by-two principal minors of $H_BH_B^{\mathsf T}-J_B$ gives
\begin{equation}\label{eq:J-pair}
 \bigl(H_B(x)H_B(y)-J_B(x,y)\bigr)^2
 \le H_B(x)^2H_B(y)^2,
\end{equation}
and hence $J_B(x,y)\le2H_B(x)H_B(y)$.

It remains to remove the assumption that \(f\) is supported on \(U_B\).
Terms involving $r$ vanish, while $H_B(x)=1$ for $x\in B$ and $a_B(x,y)\le H_B(y)$ whenever $x\in B$.
Put $a=\sum_{x\in B}f(x)$ and $b=\sum_{x\in U_B}f(x)H_B(x)$.
Decomposing the left side of \eqref{eq:ancestor} according to whether neither, exactly one, or both of \(x,y\) lie in \(B\), the three contributions are bounded by $b^2$, $2ab$, and $a^2$, respectively.
This proves \eqref{eq:ancestor} for arbitrary nonnegative $f$, and also \eqref{eq:pointwise}.
\end{proof}

\section{Proof of Theorem~\ref{thm:bk}}\label{sec:finite-apps}

We prove Theorem~\ref{thm:bk} by applying Theorem~\ref{thm:separation}.
For a finite graph $G$ and $m>0$, adjoin a cemetery vertex $\partial$ and an edge of conductance $m$ from each $v\in V$ to $\partial$.
Deleting $\partial$ from a weighted spanning tree of the augmented
graph yields a rooted forest $\cF_m$, called the
\emph{massive spanning forest} with constant killing rate $m$.
Each component is rooted at its unique vertex previously adjacent
to $\partial$.
Let $C_m(x)$ be its component containing $x$, and $\gamma_x^m$ be the path from $x$ to $\partial$, with $\partial$ omitted.
Wilson's algorithm identifies this path with the loop-erasure of the continuous-time random walk with jump rates $c_{xy}=\sum_{e=\{x,y\}}c_e$, stopped at an independent exponential time of rate $m$ (see e.g.~\cite[Chapter~4]{LP}). We write \(\mathbb P_m\) for the law of $\cF_m$.

\begin{proof}[Proof of Theorem~\ref{thm:bk}]
Fix $B\subseteq V$, and let $U,W$ be two independent uniform vertices of $V$, independent of $\cF_m$.
Applying \eqref{eq:ancestor} with root $\partial$ and $f=\one_V/n$ yields
\begin{equation}\label{eq:finite-hit-m}
 \PP_m(\LCA(U,W)\in B)
 \le \PP_m(\gamma_U^m\cap B\ne\varnothing)^2.
\end{equation}
A spanning tree of the augmented graph containing \(j\) cemetery edges has weight proportional to \(m^j\).
As $m\downarrow0$, its law therefore converges to a weighted spanning tree $T$ of $G$ with an independent uniform vertex $R$ joined to $\partial$.
The joint law of $(T,U,W,R)$ is invariant under the transitive automorphism group, so $\LCA_R(U,W)$ is uniform on $V$.
Taking \(m\downarrow0\) in \eqref{eq:finite-hit-m} gives
\[
 \frac{|B|}{n}
 \le \PP([U,R]_T\cap B\ne\varnothing)^2,
\]
which proves \eqref{eq:set-hit}.
For a partition $\mathcal P=\{B_1,\ldots,B_k\}$ of $V$, we then have $\EE N_{\mathcal P}
=\sum_{i=1}^k\PP(\gamma\cap B_i\ne\varnothing)
\ge\sum_{i=1}^k\sqrt{|B_i|/n}$; this implies \eqref{eq:partition-hit}.
Taking $\mathcal P$ to be the partition into singletons
gives $\EE|\gamma|\ge\sqrt n$.
Since $|\gamma|=d_T(X,Y)+1$, we obtain \eqref{eq:bk}.
\end{proof}

\section{Proof of Theorem~\ref{thm:3d}}\label{sec:3d}

We prove Theorem~\ref{thm:3d} using local connectivity estimates of massive spanning forests.
The construction of $\cF_m$ extends to every countable locally finite undirected graph with jump rates satisfying $\sum_yc_{xy}=1$.
Indeed, for finitely many starting points, the underlying random walks in Wilson's algorithm have finite lifetimes a.s.
Exhaustion by finite wired sets and retaining the killing rate $m$ therefore gives convergence of their entire branches and common ancestors.
Consequently, Theorem~\ref{thm:separation} passes to the infinite-volume limit for finite \(B\) and finitely supported \(f\).
We set $\LCA(x,y)=\partial$ when $x,y$ belong to distinct components.

\begin{proof}[Proof of Theorem~\ref{thm:3d}]
We work with continuous-time simple random walk with unit jump rate.
We first obtain a local connectivity estimate at killing rate $R^{-2}$, and then extend it to all \(0<m\le R^{-2}\) by monotonicity in the killing rate.

\emph{Local connectivity estimate.}
We first prove a second-moment estimate on any transitive weighted undirected graph with
$\sum_y c_{xy}=1$ for every $x\in V$, which will also be used in Section~\ref{sec:killed}.
Let $p_t(x,y)$ be its transition kernel, fix an origin $o$, and put
\begin{equation}\label{eq:massive-green}
 g_m(x,y)=\int_0^\infty e^{-mt}p_t(x,y)\dd t,
 \qquad \cB(m)=\sum_zg_m(o,z)^2.
\end{equation}
Note that
\begin{equation}\label{eq:bubble-identity}
 \sum_zg_m(x,z)=m^{-1},\qquad
 \cB(m)=\int_0^\infty t e^{-mt}p_t(o,o)\dd t.
\end{equation}
For independent walks $X,Y$ started at $x,y$, with independent exponential lifetimes $\zeta_X,\zeta_Y$ of rate $m$, define
\begin{equation}\label{eq:I-occupation}
 I_{x,y}=\int_0^{\zeta_X}\int_0^{\zeta_Y}
                 \one_{\{X_s=Y_t\}}\dd t\dd s.
\end{equation}
Writing $g=g_m$, we have
\begin{align}
 \EE I_{x,y}
 &=\sum_u g(x,u)g(y,u),\label{eq:I-first}\\
 \EE I_{x,y}^2
 &=2\sum_{u,v}g(x,u)g(y,u)g(u,v)^2
   +2\sum_{u,v}g(x,u)g(y,v)g(u,v)^2\le4\cB(m)^2.\label{eq:I-second}
\end{align}
The last inequality follows from Cauchy--Schwarz inequality and that $ab\le(a^2+b^2)/2$.
Wilson's algorithm and the inequality of~\cite[Lemma~1.2 and Remark~1.3]{LPS} imply
\begin{equation}\label{eq:LPS}
 \PP_m(y\in C_m(x))
 =\PP\bigl(\LE(X[0,\zeta_X))\cap Y[0,\zeta_Y)\ne\varnothing\bigr)
 \ge2^{-8}\PP(I_{x,y}>0).
\end{equation}

Now on $\ZZ^3$, set $\mu=R^{-2}$, and take $x,y\in B_R$.
The heat-kernel bounds give
\begin{equation}\label{eq:3d-green}
 \inf_{x,u\in B_R}g_\mu(x,u)\ge cR^{-1},
 \qquad \cB(\mu)\le CR.
\end{equation}
Here, the first estimate follows by integrating over $t\in[R^2,2R^2]$, and the second follows from
$p_t(0,0)\le C(1+t)^{-3/2}$ and \eqref{eq:bubble-identity}; see~\cite[Theorem~2.1.3]{LL}.
Combining~\eqref{eq:I-first} and \eqref{eq:I-second} gives
$\EE I_{x,y}\ge cR$ and $\EE I_{x,y}^2\le CR^2$.
Hence the second-moment inequality gives $\PP(I_{x,y}>0)\ge c$, and \eqref{eq:LPS} gives $\inf_{x,y\in B_R}\PP_\mu(y\in C_\mu(x))
\ge a_0>0$, where $a_0$ is independent of $R$.

Write \(x\longleftrightarrow_D y\) if \(x\) and \(y\) are connected in \(\mathfrak F_m\) and their unique path in \(\mathfrak F_m\) is contained in \(D\).
Doob's inequality applied up to the exponential lifetime $\zeta_X$ gives
\[
 \sup_{x\in B_R}\PP_x\left(\sup_{t<\zeta_X}|X_t|\ge KR\right)
 \le \frac{C}{(K-1)^2}.
\]
If both $X[0,\zeta_X)$ and $Y[0,\zeta_Y)$
stay in $B_{KR}$, and $Y$ hits
$\LE(X[0,\zeta_X))$ before $\zeta_Y$,
then Wilson's algorithm gives
$x\longleftrightarrow_{B_{KR}}y$.
Choose $K>1$ sufficiently large such that
$2C/(K-1)^2\le a_0/2$.
Then with $a=a_0/2$, we obtain
\begin{equation}\label{eq:local-hyp}
 \inf_{x,y\in B_R}
 \PP_\mu(x\longleftrightarrow_{B_{KR}}y)\ge a.
\end{equation}

\emph{Removing the killing.}
For $0<m\le\mu$, the massive spanning forests
can be coupled so that $\cF_\mu\subseteq\cF_m$.
To see this, orient the edges arbitrarily and define
$\nabla:\ell^2(E)\to\ell^2(V)$ by $\nabla\delta_e=\sqrt{c_e}(\delta_{e^+}-\delta_{e^-})$.
By the transfer-current theorem
\cite[Section~4.2]{LP}, the edge set of $\cF_m$
is determinantal with kernel $ \mathcal K_m=\nabla^*(\nabla\nabla^*+mI)^{-1}\nabla$.
The same formula holds in infinite volume by exhaustion; see e.g.~\cite[Section~10.1]{LP}. Now monotonicity of the resolvent gives
$0\preceq \mathcal{K}_\mu\preceq \mathcal{K}_m\preceq I$, and the determinantal domination theorem (see~\cite[Theorem~2.9]{Lyons}) therefore couples the spatial edge sets with $\cF_\mu\subseteq\cF_m$. Consequently, \eqref{eq:local-hyp} holds for every $0<m\le\mu$ as the event $\{x\longleftrightarrow_{B_{KR}}y\}$ there is increasing.

On this event, $\LCA(x,y)$ lies on their connecting path, hence in $B_{KR}$.
Put
\[
 h_m(u)=\PP_m(u\in\gamma_0^m),\qquad
 A_m(t)=\sum_{u\in B_t}h_m(u).
\]
We apply \eqref{eq:ancestor} with $f=\one_{B_R}$ and $B=\{v\}$ to each $v\in B_{KR}$, and then sum over \(v\).
Translation invariance then gives
\begin{align}
 a|B_R|^2
 &\le\sum_{v\in B_{KR}}\sum_{x,y\in B_R}
              \PP_m(\LCA(x,y)=v)\notag\\
 &\le\sum_{v\in B_{KR}}
       \left(\sum_{x\in B_R}h_m(v-x)\right)^2
 \le |B_{KR}|A_m((K+1)R)^2.\label{eq:local-bound}
\end{align}
Thus $A_m((K+1)R)\ge cR^{3/2}$, uniformly for $0<m\le R^{-2}$.
By transience, for every fixed \(u\), the indicator \(\mathbf1_{\{u\in\LE(S[0,t])\}}\) converges to  \(\mathbf1_{\{u\in\eta\}}\) as \(t\to\infty\).
Consequently, $h_m(u)\to\PP(u\in\eta)$ as $m\downarrow0$.
Taking $m\downarrow0$ in the finite sum $A_m((K+1)R)$ and rescaling $R$ by the fixed factor $K+1$ proves \eqref{eq:occupation-main} after adjusting the constant \(c\).

\emph{The growth exponent.}
For independent simple random walks $S^1,S^2$ started at $0$, define
\begin{equation}\label{eq:Es}
 \Es(r)=\PP\bigl(\LE(S^1[0,\tau_r^1])
                    \cap S^2[1,\tau_r^2]=\varnothing\bigr),
\end{equation}
where $\tau_r^i$ is the first exit time of $S^i$ from $B_r$.
By~\cite[Corollary~1.3]{LS}, we have
\begin{equation}\label{eq:Es-beta}
 \Es(r)\asymp r^{\beta-2},\qquad r\ge1.
\end{equation}
The finite-volume one-point estimate in~\cite[Proposition~8.2]{ShiG}, combined with transience and the Green-function bound $\sum_{j\ge0}\PP_0(S_j=z)\le C/(1+|z|)$~\cite[Theorem~4.3.1]{LL}, implies
\begin{equation}\label{eq:one-point-infinite}
 \PP(z\in\eta)\le\frac{C}{1+|z|}\Es(|z|),
 \qquad z\ne0.
\end{equation}
Summing over lattice thus yields $\EE|\eta\cap B_R|\le CR^\beta$.
Combined with \eqref{eq:occupation-main} gives $\beta\ge3/2$.
\end{proof}

\section{Further applications}\label{sec:further}

\subsection{Intersection of independent infinite LERWs}
\label{subsec:independent-lerw}

In~\cite[Question~6.3]{LPS}, the authors asked whether
loop-erasures of two independent copies of a transient Markov chain preserve a.s.~infinite intersection. Theorem~\ref{thm:3d} gives the three-dimensional case.

\begin{theorem}\label{thm:independent-lerw}
Let $\eta_a$ and $\eta_b$ be two independent infinite LERWs on $\ZZ^3$, started at arbitrary $a,b\in\ZZ^3$. Then
\[
 \PP\bigl(|\eta_a\cap\eta_b|=\infty\bigr)=1.
\]
\end{theorem}

\begin{proof}
Put $q=3-\beta$ and $\kappa=3-2q=2\beta-3\ge0$.
By \eqref{eq:Es-beta} and \eqref{eq:one-point-infinite}, $\PP(z\in\eta_x)\le C(1+|z-x|)^{-q}$.
For $n\ge1$, let
\[
 D_n(x)=\{z\in\ZZ^3:
 |z-x-(n/2)e_1|_\infty<n/1000\},
 \qquad e_1=(1,0,0),
\]
and let $\sigma_n^x$ be the first exit time of $\eta_x$
from $B(x,n):=x+B_n$.

Choose $n_0\ge10^5(1+|a|+|b|)$ and set $n_j=2^jn_0$.
The sets $Q_j:=D_{n_j}(a)\cap D_{n_j}(b)$ are pairwise disjoint, satisfy $|Q_j|\asymp n_j^3$,
and lie at distance comparable to $n_j$ from both
$a$ and $b$. Define
\[
 Z_j=\bigl|Q_j\cap\eta_a[0,\sigma_{n_j}^a]
                    \cap\eta_b[0,\sigma_{n_j}^b]\bigr|.
\]
By \cite[Proposition~3.16]{ACHS}, applied with
$K=\varnothing$, for $z\ne w$ in $D_n(x)$,
\[
 \PP(z\in\eta_x[0,\sigma_n^x])\ge cn^{-q},
 \qquad
 \PP(z,w\in\eta_x[0,\sigma_n^x])
 \le Cn^{-q}|z-w|^{-q}.
\]
Independence therefore gives $\EE Z_j\asymp n_j^\kappa$ and $\EE Z_j^2
 \le Cn_j^\kappa
 \left(1+\sum_{r=1}^{n_j}r^{2-2q}\right)
 \le Cn_j^{2\kappa}(1+\log n_j)$,
where the last inequality uses $2-2q=\kappa-1$
and $\kappa\ge0$.

For $i<j$, $z\in Q_i$, $w\in Q_j$, and
$x\in\{a,b\}$, we have
\[
 |z-x|\asymp n_i,\qquad
 |w-x|\asymp|z-w|\asymp n_j.
\]
Applying \cite[Proposition~8.2, case~(I), $k=2$]{ShiG}
in $B(x,N)$ and letting $N\to\infty$ yield
\[
 \PP(z,w\in\eta_x)
 \le C\bigl(n_i^{-q}n_j^{-q}+n_j^{-2q}\bigr)
 \le Cn_i^{-q}n_j^{-q}.
\]
Consequently, $\EE(Z_iZ_j)\le Cn_i^\kappa n_j^\kappa$.

Set $Y_j=n_j^{-\kappa}Z_j$ and $W_K=\sum_{j=1}^K Y_j$.
The preceding bounds imply
\[
 \EE Y_j\asymp1,\qquad
 \EE Y_j^2\le C(1+\log n_j),\qquad
 \EE(Y_iY_j)\le C\quad(i<j).
\]
Since $n_j=2^jn_0$, we have
\[
 \EE W_K\ge cK,\qquad
 \EE W_K^2\le C(K^2+K\log n_0).
\]
Thus, for $K\ge1+\log n_0$, the second-moment
inequality gives $\PP(W_K\ge cK)\ge c_0>0$,
with constants independent of $a,b$.
Since $\kappa\ge0$ and the sets $Q_j$ are disjoint, we have $W_K\le\sum_{j=1}^K Z_j
 \le|\eta_a\cap\eta_b|$.
Letting $K\to\infty$ yields
\begin{equation}\label{eq:infinite-overlap-positive}
 \PP\bigl(|\eta_a\cap\eta_b|=\infty\bigr)\ge c_0.
\end{equation}

Finally, note that three-dimensional simple random walk has
infinitely many cut times a.s.~\cite[Theorem~1.2]{LawlerCut}.
Consequently, removing any finite initial segment
of an infinite random walk changes its loop-erasure by only
finitely many vertices.
Let $A=\{|\eta_a\cap\eta_b|=\infty\}$, and let
$\mathcal F_n=\sigma(S_j^a,S_j^b:0\le j\le n)$,
where $S^a,S^b$ are the underlying independent random walks.
The Markov property and \eqref{eq:infinite-overlap-positive} imply $\PP(A\mid\mathcal F_n)
 =\PP\bigl(
   |\eta_{S_n^a}\cap\eta_{S_n^b}|=\infty
   \bigr)
 \ge c_0$.
By martingale convergence,
$\PP(A\mid\mathcal F_n)\to\one_A$ a.s.
Therefore, $\PP(A)=1$.
\end{proof}

\subsection{Branch length and component size in the massive spanning forest}\label{sec:killed}

Let $G=(V,E)$ be a finite or countably infinite locally
finite graph with a transitive conductance-preserving
automorphism group $\Gamma$ satisfying
\begin{equation}\label{eq:MT}
 \sum_y F(o,y)=\sum_x F(x,o)
\end{equation}
for every nonnegative diagonally $\Gamma$-invariant
function $F$ on $V\times V$.
Note that \eqref{eq:MT} holds for finite transitive graphs and for $\ZZ^d$
by translation invariance.

Assume $\sum_y c_{xy}=1$ for every $x\in V$, and let
$\cF_m$ be the massive spanning forest defined above,
using exhaustion in the infinite case.
Fix $o\in V$ and set
\[
 \chi(m)=\EE|C_m(o)|,\qquad L(m)=\EE|\gamma_o^m|.
\]
With $g_m$ and $\cB(m)$ as in \eqref{eq:massive-green},
we obtain the following bounds.

\begin{proposition}\label{thm:return}
For every $m>0$, we have
\begin{equation}\label{eq:return-bound}
 \frac{1}{2^{10}m^2\cB(m)}
 \le\chi(m)\le L(m)^2,
 \qquad
 \frac{1}{32m\sqrt{\cB(m)}}\le L(m)\le1+m^{-1}.
\end{equation}
\end{proposition}
\begin{proof}
We first prove $\chi(m)\le L(m)^2$.
For each $v\in V$ and each finite set $F\subseteq V$, apply
\eqref{eq:ancestor} with $B=\{v\}$ and $f=\one_F$.
Letting $F\uparrow V$ and using monotone convergence gives
\[
 A(v):=\sum_{x,y\in V}\PP_m(\LCA(x,y)=v)
 \le
 \left(\sum_{x\in V}\PP_m(v\in\gamma_x^m)\right)^2.
\]
Applying \eqref{eq:MT} to $F(x,v)=\PP_m(v\in\gamma_x^m)$ and to
$F(x,v)=\sum_y\PP_m(\LCA(x,y)=v)$ gives
\[
 \sum_x\PP_m(o\in\gamma_x^m)=L(m),\qquad
 A(o)=\sum_{v,y}\PP_m(\LCA(o,y)=v)=\chi(m).
\]
By Tonelli's theorem, both identities hold in \([0,\infty]\).
It follows that $\chi(m)\le L(m)^2$.
The expected number of jumps before killing is $m^{-1}$, so $L(m)\le1+m^{-1}$.

For the lower bound, write $I_y=I_{o,y}$ in \eqref{eq:I-occupation}.
Summing \eqref{eq:I-first}--\eqref{eq:I-second} over $y$, using \eqref{eq:bubble-identity} and transitivity implies $\sum_y\EE I_y=m^{-2}$ and $\sum_y\EE I_y^2=4m^{-2}\cB(m)$.
Cauchy--Schwarz inequality, first on finite sets and then by exhaustion, gives
\[
 \sum_y\PP(I_y>0)
 \ge
 \frac{(\sum_y\EE I_y)^2}{\sum_y\EE I_y^2}
 =\frac{1}{4m^2\cB(m)}.
\]
Summing \eqref{eq:LPS} yields the lower bound on $\chi(m)$.
The lower bound on $L(m)$ follows from $\chi(m)\le L(m)^2$.
\end{proof}

Proposition~\ref{thm:return} has the following consequence for the long-range LERW on $\ZZ^3$.
\begin{proposition}\label{prop:long-range}
For $3/2<\alpha<2$, let $S$ be a symmetric random walk on $\ZZ^3$
started at $0$, with jump distribution $q(z)\asymp |z|^{-3-\alpha}$. Let $\tau_R=\inf\{n\ge0:S_n\notin B_R\}$.
Then
\begin{equation}\label{eq:long-range-growth}
 \beta_\alpha:=\liminf_{R\to\infty}
 \frac{\log\EE|\LE(S[0,\tau_R])|}{\log R}\in\left[\frac32,\alpha\right].
\end{equation}
In particular, $\beta_\alpha\to3/2$ as $\alpha\downarrow3/2$. The same results hold with $\limsup$ in place of $\liminf$.
\end{proposition}

\begin{proof}
For $r>0$, let $X^{(r)}$ be the continuous-time random walk started
at $0$ with jump rates
$q_r(z)=q(z)\mathbf1_{\{|z|\le r\}}$, killed at an
independent exponential time $\zeta$ of rate
$m_r=\sum_{|z|>r}q(z)\asymp r^{-\alpha}$.
Write $L_r=|\LE(X^{(r)}[0,\zeta))|$.
Then the corresponding $\cB(m_r)$ is bounded by a constant multiple
of that for $S$.
By Proposition~\ref{thm:return} and the heat-kernel estimates~\cite[Theorem~1.1]{BL},  $\EE L_r$ is thus at least
$c r^{3/2}$ for some $c>0$.

We claim that restricting to random walk trajectories contained in $B_{Ar\log r}$ preserves this lower bound for sufficiently large $A$. Indeed, an exponential supermartingale argument
(cf.~\cite[Proposition~2.5]{BL}) gives
\[
 \PP\left(\sup_{t<\zeta}|X_t^{(r)}|
       \ge Ar\log r\right)\le Cr^{-cA}.
\]
Since the number $N_\zeta$ of jumps before killing is geometric
with parameter $m_r$, we have
$\EE L_r^2\le\EE(N_\zeta+1)^2
\le Cr^{2\alpha}$. Cauchy--Schwarz inequality then
implies that trajectories leaving $B_{Ar\log r}$
contribute $o(r^{3/2})$ to $\EE L_r$ for sufficiently
large $A$.

We view $S$ as a continuous-time random walk with unit jump rate.
Its jumps longer than $r$ occur at rate $m_r$,
independently of the shorter jumps, so its trajectory
before the first such jump has the law of
$X^{(r)}[0,\zeta)$.
This jump exceeds $2R$ with probability $m_{2R}/m_r$.
If the preceding trajectory stays in $B_{R/2}$,
the jump exits $B_R$ without erasing its loop-erasure. Hence, by taking $r=\lfloor R/(4A\log R)\rfloor$,
we obtain
\[
 \EE|\LE(S[0,\tau_R])|
 \ge \frac{m_{2R}}{m_r}
 \EE\left[L_r;\sup_{t<\zeta}|X_t^{(r)}|<R/2\right]
 \ge \frac{cR^{3/2}}{(\log R)^{\alpha+3/2}},
\]
where the last inequality follows from the previous paragraph. Finally, since $\EE|\LE(S[0,\tau_R])|\le1+\EE\tau_R\le CR^\alpha$, we conclude.
\end{proof}

We refer the reader to Xiao--Pan--Fan--Deng~\cite{XPFD}
for numerical estimates of $\beta_\alpha$.

\begin{remark}
The inequality \eqref{eq:pointwise} also extends to the wired uniform spanning forest (WUSF) of any transient locally finite graph. Indeed, applying \eqref{eq:pointwise} to the wired spanning trees via exhaustion and using Wilson's algorithm rooted at infinity, we obtain
\[
 \PP(x\leftrightarrow y\text{ in the WUSF})
 \le 2\sum_{v\in V}
       \PP(v\in\eta_x)\PP(v\in\eta_y),
\]
where $\eta_x$ and $\eta_y$ denote independent infinite LERWs started at $x$ and $y$.
\end{remark}

\smallskip
\paragraph{\textbf{Acknowledgments}}
The author thanks Tom Hutchcroft, Xinyi Li, and Runsheng Liu for helpful conversations. This work was supported by the National Key R\&D Program of China (No.~2021YFA1002700) and Beijing Natural Science Foundation (JQ26001).

\end{document}